\documentclass[11pt,a4paper,reqno]{amsart}
\usepackage{float, verbatim}
\usepackage{amsmath}
\usepackage{amssymb}
\usepackage{amsthm}
\usepackage{amsfonts}
\usepackage{mathtools}
\numberwithin{equation}{section}
\usepackage[colorlinks = true, citecolor = black, pagebackref=false]{hyperref}
\usepackage{color}
\usepackage{xcolor}
\usepackage{enumerate}
\usepackage{cite}

\newcommand{\Pp}{\mathbb{P}}
\newcommand{\N}{\mathbb{N}}
\newcommand{\supp}{\operatorname{supp}}

\newtheorem*{thm*}{Theorem}

\newtheorem{thm}{Theorem}[section]

\newtheorem{lem}[thm]{Lemma}
\newtheorem{coro}[thm]{Corollary}
\newtheorem{defi}[thm]{Definition}

\newtheorem{prop}[thm]{Proposition}
\newtheorem{conj}[thm]{Conjecture}
\newtheorem{rem}[thm]{Remark}
\newtheorem{ques}[thm]{Question}

\newcommand{\R}{\mathbb{R}}
\newcommand{\Z}{\mathbb{Z}}

\usepackage{todonotes}

\makeatletter
\@namedef{subjclassname@2020}{\textup{2020} Mathematics Subject Classification}
\makeatother

\begin{document}

\title[]{A conditional arithmetic obstruction to the prime numbers
as a spectrum of a probability measure}

\author{Zhi-Yi Wu}
\address{Z.-Y. Wu, School of Mathematics and Information Science, Guangzhou University, Guangzhou, 510006, P.~R.~China}
\curraddr{}
\email{zhiyi\_wu2021@163.com}
%\thanks{*Corresponding author.}
\author{Qian Zhao}
\address{Q. Zhao, School of Mathematics and Information Science, Guangzhou University, Guangzhou, 510006, P.~R.~China}
\curraddr{}
\email{qian\_zhao2026@163.com}
\thanks{ }

\subjclass[2020]{28A80,11N05,42C30}

\keywords{spectrum; prime numbers; Polignac conjecture}

\date{}

\dedicatory{}

\begin{abstract}
Let $\Pp=\{2,3,5,7,\ldots\}$ denote the set of prime numbers.
We prove that if every sufficiently large positive even integer can be
represented as a difference of two primes, then there is no Borel
probability measure $\mu$ on $\R$ for which \(\left\{e^{2\pi i p x}:p\in\Pp\right\}\)
is an orthonormal basis of $L^2(\mu)$. In particular, under the Polignac conjecture, the prime numbers $\Pp$ cannot be a spectrum (i.e., the set of frequencies of an exponential orthonormal basis) of any probability measure on $\R$.
\end{abstract}

\maketitle
\allowdisplaybreaks
\section{Introduction}
A Borel probability measure $\mu$ with compact support on $\R^d$ is called a
\emph{spectral measure} if there exists a countable set
$\Lambda\subseteq\R^d$ such that
\(E(\Lambda):=\left\{
        e_\lambda(x)=e^{2\pi i\lambda x}:
        \lambda\in\Lambda\right\}\)
is an orthonormal basis for $L^2(\mu)$. In this case, $\Lambda$ is called
a {\it spectrum} of $\mu$.

The study of spectral measures originates from the classical fact that the
Lebesgue measure on $[0,1]$ is spectral with spectrum $\Z$ and was
initiated by Fuglede's conjecture \cite{Fu}, which asserted that the normalized
Lebesgue measure on a set $\Omega$ of positive finite Lebesgue measure
is a spectral measure if and only if $\Omega$ tiles $\R^d$ by
translations.  Although the conjecture turned out to be false in
dimensions $d\ge 2$ (see \cite{Z26} and the references therein), the theory of spectral measures has continued to
flourish in many unexpected directions.  A landmark was the discovery
by Jorgensen and Pedersen \cite{JoPe} that certain singular, self-similar Cantor
measures admit spectra.  Since then, the study of the spectrality of
singular measures has become an active research field \cite{LW02,Dai12,DHLai19}, driving the
construction of new spectral measures and raising the general question:
given a singular continuous measure, what are its possible spectra?  Conversely, given a
discrete set, when can it be realized as a spectrum?

A fundamental structural result for spectral measures establishes a trichotomy on the lower Beurling density of spectra, proven by He, Lai and Lau \cite{HeLaLa}. To state this result, we first recall the definition of lower Beurling density, which describes the minimal asymptotic point count of a discrete set inside any large translated interval.

\begin{defi}\label{def:density}
\rm For a discrete set $\Lambda\subseteq\mathbb R$, the {\it lower Beurling density}
is defined as
\[
D^{-}(\Lambda)=\liminf_{h\to\infty}\inf_{x\in\mathbb R}
\frac{\#\bigl(\Lambda\cap[x-h,x+h]\bigr)}{h},
\]
where $\#$ denotes the cardinality of a set.
\end{defi}

\begin{prop}[\cite{HeLaLa}]\label{prop:density}
Let \(\mu\) be a compactly supported probability measure on
\(\mathbb{R}^d\), and let \(\Lambda\) be a spectrum of \(\mu\).
\begin{enumerate}[(i)]
    \item  If $\mu = \sum_{c \in \mathcal{C}} p_c \delta_c$ is discrete, then $\#\Lambda < \infty$ and $\#\mathcal{C} < \infty$.
    \item If \(\mu\) is singular continuous, then \(D^-(\Lambda) = 0\).
    \item If \(\mu\) is absolutely continuous, then \(D^-(\Lambda) > 0\).
\end{enumerate}
\end{prop}

As a consequence of Proposition~\ref{prop:density}, every spectral measure is of pure type, i.e., it is either a finite sum of Dirac masses,
absolutely continuous, or singular continuous with
respect to the Lebesgue measure.  Therefore, each type can be
studied separately, and the most subtle case is the singular
continuous one.

  Given a singular continuous spectral measure $\mu$, every spectrum of $\mu$ has zero lower Beurling density by Proposition \ref{prop:density}(ii). This fact naturally leads to an inverse problem: can every sufficiently sparse discrete set be realized as the spectrum of some singular continuous measure? More precisely,

\begin{ques}\label{ques:inverse}
Given a discrete set \(\Lambda\subseteq\R\) with \(D^-(\Lambda)=0\), does
there exist a (necessarily singular continuous) probability measure with compact support
\(\mu\) for which \(\Lambda\) serves as a spectrum?
\end{ques}

The most prominent example among the Cantor-type spectral measures discovered by Jorgensen and Pedersen is the standard 4-th Cantor measure $\mu_4$\cite{JoPe}. It is a
singular continuous spectral measure, with a canonical spectrum given by
\begin{equation*}%\label{eq:Lambda4}
    \Lambda_4
    =\Biggl\{\sum_{k=0}^{n} a_k 4^k : a_k\in\{0,1\},\ n\in\N\Biggr\}.
\end{equation*}
Since $\mu_4$ is singular continuous, Proposition~\ref{prop:density}
implies $D^{-}(\Lambda_4)=0$, which can also be verified directly from
the lacunary structure of $\Lambda_4$.  Subsequently, many spectra  were obtained for some self-similar,
self-affine and Moran spectral measures \cite{AH14,Dai16,DHLai13,DHS09,DJ07a}; in all
these cases the underlying measures are singular continuous, so their
spectra likewise have zero lower Beurling density.  A common feature
of these constructions is that the spectra are highly lacunary and tree-like sets with extremely sparse additive structure.

Despite the numerous explicit examples outlined above, Question~\ref{ques:inverse} remains
largely open for sets that are not generated by such iterative constructions.
The primes naturally come to mind as a primary candidate.
Beyond their intrinsic importance within number theory, the set of prime numbers
\[
    \Pp = \{2,3,5,7,\ldots\},
\]
has zero lower Beurling density (see Lemma \ref{lem:Pdensity}).
It is therefore a particularly natural test case for Question \ref{ques:inverse}.
Our main result establishes that, under a mild arithmetic hypothesis concerning prime gaps,
the answer is negative.

We first establish that any Borel probability measure with spectrum $\Pp$ must indeed be singular continuous.
\begin{prop}\label{prop:singular}
If $\Pp$ is the spectrum of a Borel probability measure $\mu$, then $\mu$ must be singular continuous.
\end{prop}

Then the lower Beurling density restrictions mentioned above are automatically
satisfied and the obstacle, if any, must be of an arithmetic nature. Our main result requires only the following weak hypothesis on the set of prime differences.

\medskip

\noindent
\textbf{Eventual even difference hypothesis.}
There exists $H\in\mathbb{N}$ such that
\begin{equation}\tag{ED}\label{eq:ED}
    2h\in \Pp-\Pp
    \qquad\text{for all } h\geq H.
\end{equation}
Here
\[
    \Pp-\Pp
    :=
    \{p-q : p,q\in\Pp\}
\]
denotes the prime difference set. Our main result is the following.

\begin{thm}\label{thm:main}
Suppose the eventual even difference hypothesis \textup{(ED)} holds. Then
there is no Borel probability measure $\mu$ on $\R$ such that
$E(\Pp)$ is an orthonormal basis for $L^2(\mu)$.
\end{thm}

The condition \textup{(ED)} is a weak hypothesis that only requires
every sufficiently large even integer to be represented as a difference
of two primes at least once.  A much stronger statement in additive number theory is the celebrated Polignac conjecture, originally posed by Alphonse de Polignac in 1849 \cite{Po}.

\begin{conj}[Polignac conjecture]\label{conj:polignac}
For every positive even integer $2h$, there exist infinitely many pairs
of consecutive primes $(p,q)$ with $p-q=2h$.
\end{conj}
\begin{rem}
The case \(h=1\) recovers the well-known twin prime conjecture. The full conjecture remains unresolved to this day.
\end{rem}

As an immediate corollary, Theorem~\ref{thm:main} yields a conditional resolution of our prime spectrum problem.

\begin{coro}\label{cor:polignac}
If the Polignac conjecture is true, then the set of primes cannot be a spectrum of any Borel probability measure on \(\mathbb{R}\). In particular, there is no singular continuous Borel probability
measure $\mu$ for which $E(\Pp)$ is an orthonormal basis for
$L^2(\mu)$.
\end{coro}

\section{A necessary condition: singular continuity}
In this section, we prove Proposition \ref{prop:singular}.  We begin with a preliminary lemma concerning the lower Beurling density of the set of primes.
\begin{lem}\label{lem:Pdensity}
$D^{-}(\mathbb{P})=0$.
\end{lem}

\begin{proof}
A classical elementary fact is that the set of prime numbers contains arbitrarily long gaps, i.e., for any integer $n\ge 2$, the consecutive integers
\[
n!+2,\; n!+3,\; \dots,\; n!+n
\]
are all composite; see, for example, the proof of Theorem~5 in \cite[\S1.4, p.~5]{HW08}. Hence the open interval $(n!+1,\, n!+n+1)$ contains no primes and has length $n$.

Fix an arbitrary $h>0$. Choose an integer $n$ with $n > 2h+2$ and consider the interval
$I = [n!+2,\, n!+n]$, whose length is $n-1>2h$.
Let $x_0 = n! + \frac{n+2}{2}$ be the midpoint of $I$.
Then $[x_0-h,\, x_0+h] \subseteq I$, and consequently
\[
\#\bigl(\mathbb{P} \cap [x_0-h,\, x_0+h]\bigr) = 0.
\]
Thus, for every $h>0$,
\[
\inf_{x\in\mathbb{R}} \frac{\#\bigl(\mathbb{P}\cap[x-h,x+h]\bigr)}{h} = 0.
\]
Taking $\liminf$ as $h\to\infty$ yields $D^{-}(\mathbb{P}) = 0$.
\end{proof}

\begin{proof}[Proof of Proposition~\ref{prop:singular}]
By the Lebesgue-Radon-Nikodym theorem, the measure $\mu$ can be written uniquely as
\[
    \mu = \mu_{\mathrm{d}} + \mu_{\mathrm{ac}} + \mu_{\mathrm{sc}},
\]
where $\mu_{\mathrm{d}}$ is discrete, $\mu_{\mathrm{ac}}$ is absolutely continuous with respect to the Lebesgue measure $\mathcal{L}$, and $\mu_{\mathrm{sc}}$ is singular
continuous, i.e., $\mu_{\mathrm{sc}}$ is continuous but $\mu_{\mathrm{sc}}\perp \mathcal{L}$. We will show that both \(\mu_{\mathrm{d}} = 0\) and \(\mu_{\mathrm{ac}} = 0\), which forces \(\mu = \mu_{\mathrm{sc}}\).

We first prove $\mu$ has no atoms. Suppose to the contrary that $\mu_{\mathrm{d}}\neq0$.  Then there
exists a point $a\in\R$ with $c:=\mu(\{a\})>0$. Consider the
indicator function $f:=\mathbf{1}_{\{a\}}$,
i.e., $f(x)=1$ if $x=a$ and $f(x)=0$ otherwise.  Since $\mu(\{a\})=c$, we have
\(\|f\|_{L^2(\mu)}^2 = c\). For every $p\in\Pp$,
\[
    \langle f,e_p\rangle_{L^2(\mu)}
    = \int_{\R} \mathbf{1}_{\{a\}}(x)\,e^{-2\pi i p x}\,\mathrm{d}\mu(x)
    = c\,e^{-2\pi i p a},
\]
and so $|\langle f,e_p\rangle|^2 = c^2$.

Because $E(\Pp)$ is an orthonormal basis of $L^2(\mu)$, Parseval's
identity gives
\begin{equation}\label{eqeqrrr}
    c = \|f\|_{L^2(\mu)}^2
    = \sum_{p\in\Pp} \bigl|\langle f,e_p\rangle\bigr|^2
    = \sum_{p\in\Pp} c^2
    = c^2\cdot\#\Pp.
\end{equation}
Since $\Pp$ is infinite, the right-hand side of \eqref{eqeqrrr} is infinite whenever
$c>0$, contradicting the finiteness of the left-hand side.  Therefore
$\mu_{\mathrm{d}}=0$.

Next, we prove that $\mu$ has no absolutely continuous part. Suppose to the contrary that $\mu_{\mathrm{ac}}\neq0$.  Write
$\mathrm{d}\mu_{\mathrm{ac}} = w\,\mathrm{d}x$ with
$w\in L^1(\R)$, $w\ge0~\mathcal{L}-a.e.$, and $w\not\equiv0$.  Since
$\mu_{\mathrm{sc}}\perp\mathcal{L}$, there exists a Borel set
$S\subseteq\R$ such that $\mathcal{L}(S)=0$ and
$\mu_{\mathrm{sc}}(\R\setminus S)=0$.

Since $w\not\equiv0$, there exists $\varepsilon>0$ such that
$\mathcal{L}(\{x:w(x)>\varepsilon\})>0$.  By Chebyshev's inequality,
\[
    \mathcal{L}(\{x:w(x)>\beta\}) \le \frac{1}{\beta}\int_{\R}w(x)\,\mathrm{d}x
    \to 0 \quad\text{as }\beta\to\infty.
\]
Hence we can pick $\beta>\varepsilon$ large enough so that the set
$B:=\{x\in\R\setminus S:\varepsilon<w(x)\le\beta\}$ satisfies
$\mathcal{L}(B)>0$.  Since $B\cap[-R,R]\uparrow B$ as $R\to\infty$,
continuity of measure gives $\mathcal{L}(B\cap[-R,R])\to\mathcal{L}(B)$.
Choosing $R>0$ sufficiently large, we obtain
\[
    A:=B\cap[-R,R]=\bigl\{x\in[-R,R]\setminus S : \varepsilon\le w(x)\le\beta\bigr\}
\]
with $\mathcal{L}(A)>0$.  Setting $\alpha:=\varepsilon$,
we have $0<\alpha\le w(x)\le\beta<\infty$ for all $x\in A$.
By construction $\mu_{\mathrm{sc}}(A)=0$ and $\mu_{\mathrm{d}}(A)=0$, so
$\mu|_A = w\,\mathrm{d}x|_A$ with $\alpha\le w\le\beta$ on $A$.

For any $h\in L^2(A)$, define
\[
    f(x) =
    \begin{cases}
        h(x)/w(x), & x\in A,\\[2pt]
        0, & x\notin A.
    \end{cases}
\]
Then
\[
    \|f\|_{L^2(\mu)}^2
    = \int_A \frac{|h(x)|^2}{w(x)}\,\mathrm{d}x,
\]
and consequently
\begin{equation}\label{eqfrede}
    \frac{1}{\beta}\,\|h\|_{L^2(A)}^2
    \le \|f\|_{L^2(\mu)}^2
    \le \frac{1}{\alpha}\,\|h\|_{L^2(A)}^2.
\end{equation}

Moreover,
\[
    \langle f,e_p\rangle_{L^2(\mu)}
    = \int_A \frac{h(x)}{w(x)}\,e^{-2\pi i p x}\,w(x)\,\mathrm{d}x
    = \int_A h(x)\,e^{-2\pi i p x}\,\mathrm{d}x
    = \widehat{h\mathbf{I}_A}(p).
\]
Since $E(\Pp)$ is an orthonormal basis of $L^2(\mu)$, Parseval's
identity yields, for every $h\in L^2(A)$,
\[
    \sum_{p\in\Pp} \bigl|\widehat{h\mathbf{I}_A}(p)\bigr|^2
    = \|f\|_{L^2(\mu)}^2.
\]
Then, by \eqref{eqfrede}, $E(\Pp) = \{e^{2\pi i p x}\}_{p\in\Pp}$ is a Fourier frame for
$L^2(A)$.

By Landau's density theorem for Fourier frames~\cite{La},
this forces
\[
    D^{-}(\Pp) \ge \mathcal{L}(A) > 0.
\]
But Lemma~\ref{lem:Pdensity} gives $D^{-}(\Pp)=0$, a contradiction.
Hence $\mu_{\mathrm{ac}}=0$.

Hence, $\mu=\mu_{\mathrm{sc}}$ is singular continuous, which completes the proof.
\end{proof}

\section{Fourier analysis and the main obstruction}
In this section we complete the proof of Theorem~\ref{thm:main}.
We first derive the Fourier identities imposed by the assumption
that $E(\Pp)$ is an orthonormal basis, and then show that these identities are incompatible
with the arithmetic hypothesis \eqref{eq:ED}.

\subsection{Fourier consequences of orthogonality and completeness}

Let $\mu$ be a Borel probability measure on $\R$. The Fourier transform of $\mu$ is
\[
    \widehat{\mu}(\xi)=
    \int_{\R}e^{-2\pi i\xi x}\,\mathrm{d}\mu(x).
\]
Since $\mu$ is a positive measure, it follows that $\widehat{\mu}(-\xi)=\overline{\widehat{\mu}(\xi)}$.

For $k\in\Z$, define
\begin{equation}\label{eq:ak}
    a(k):=|\widehat{\mu}(k)|^2.
\end{equation}
Then $a(k)\ge0$ for all $k\in\Z$, with $a(0)=1$ and $a(-k)=a(k)$.

\begin{prop}\label{prop:tiling}
Suppose that $E(\Pp)$ is an orthonormal basis of $L^2(\mu)$, and let
$a(k)$ be defined as in \eqref{eq:ak}. Then
\begin{itemize}
\item[(i)] $a(p-q) = 0$ for all $p,q\in\Pp$ with $p\neq q$;
\item[(ii)] $\displaystyle\sum_{p\in\Pp} a(n-p) = 1$ for every $n\in\mathbb{Z}$.
\end{itemize}
\end{prop}

\begin{proof}
For distinct primes $p,q$, the orthogonality of $E(\Pp)$ gives
\[
    0
    =
    \langle e_p,e_q\rangle_{L^2(\mu)}
    =
    \int_{\R}e^{2\pi i(p-q)x}\,\mathrm{d}\mu(x)
    =
    \widehat{\mu}(q-p).
\]
It follows that $a(p-q)=|\widehat{\mu}(p-q)|^2=0$, which proves (i).

Since $\mu$ is a probability measure, $\|e_n\|=1$ for all $n\geq1$. Because $E(\Pp)$ is an orthonormal basis, Parseval's identity gives
\[
    1
    =\|e_n\|^2=
    \sum_{p\in\Pp}
    \left|
        \langle e_n,e_p\rangle_{L^2(\mu)}
    \right|^2,
\]
for all $n\geq1$.

Moreover,
\[
    \langle e_n,e_p\rangle_{L^2(\mu)}
    =
    \int_{\R}e^{2\pi i(n-p)x}\,\mathrm{d}\mu(x)
    =
    \widehat{\mu}(p-n).
\]
Therefore,
\[
    1
    =
    \sum_{p\in\Pp}
    |\widehat{\mu}(p-n)|^2
    =
    \sum_{p\in\Pp}a(n-p),
\]
where the last equality uses $a(-k)=a(k)$. This proves (ii).
\end{proof}

The second assertion in Proposition \ref{prop:tiling}  can be written formally as
$\mathbf{1}_{\Pp}*\{a(k)\}\equiv 1$ on $\Z$.
Thus the modulus squares of the Fourier coefficients form a
nonnegative weighted translational tiling complement to the prime
numbers. Orthogonality simultaneously imposes the support restriction
\[
    \supp(\{a(k)\})\cap
    \bigl((\Pp-\Pp)\setminus\{0\}\bigr)
    =
    \varnothing,
\]
where
\[\supp(\{a(k)\}):=\{k\in\Z:a(k)\neq0\}\]
The contradiction in Theorem~\ref{thm:main} results from the
incompatibility of these two requirements.

\subsection{An arithmetic obstruction}
The following theorem isolates the purely combinatorial core of the
argument: it treats an arbitrary nonnegative even function
$a:\Z\to[0,\infty)$ subject only to the conditions forced by
orthogonality and completeness.  It also makes explicit that only eventual
coverage of the even integers is needed.

\begin{thm}\label{thm:obstruction}
Let $a:\Z\to[0,\infty)$ satisfy
\begin{align}
    a(0)&=1,
    \label{eq:cond1}\\
    a(-k)&=a(k),
    &&k\in\Z,
    \label{eq:cond2}\\
    a(p-q)&=0,
    &&p,q\in\Pp,\quad p\neq q,
    \label{eq:cond3}\\
    \sum_{p\in\Pp}a(n-p)&=1,
    &&n\in\Z.
    \label{eq:cond4}
\end{align}
Then the set
\[
    \{2h:h\in\N\}\setminus(\Pp-\Pp)
\]
is infinite. In particular, no function $a:\mathbb{Z}\to [0,\infty)$
can satisfy~\eqref{eq:cond1}--\eqref{eq:cond4} if every sufficiently
large positive even integer belongs to $\Pp-\Pp$.
\end{thm}

\begin{proof}
We prove it by contradiction. Assume the set
\[
E:=\{2h:h\in\N\}\setminus(\Pp-\Pp)
\]
is finite. Then there exists $H\in\N$ such that for all $h\ge H$,
\begin{equation}\label{eq:2hinPP}
    2h\in\Pp-\Pp.
\end{equation}

Since $\Pp-\Pp$ is symmetric, \eqref{eq:cond3} together
with~\eqref{eq:2hinPP} yields
\begin{equation}\label{eq:a2hzero}
    a(2h)=a(-2h)=0,
\end{equation}
for all $h\ge H$.

Now let $n$ be a negative odd integer with
\begin{equation}\label{eq:nbound}
    n\le 3-2H.
\end{equation}
Applying the tiling identity~\eqref{eq:cond4} at $n$ and separating
the prime $p=2$, we obtain
\begin{equation}\label{eq:tilingAtn}
    1
    =\sum_{p\in\Pp}a(n-p)
    =a(n-2)+\sum_{\substack{p\in\Pp\\ p\ge 3}}a(n-p).
\end{equation}

Every prime \(p\ge 3\) is odd, and \(n\) is also odd, so the difference \(p-n\) is a positive even integer. Moreover, using~\eqref{eq:nbound},
\[
    p-n\ge 3-n\ge 2H.
\]
Thus $p-n=2h$ for some $h\ge H$, and~\eqref{eq:a2hzero} gives
$a(n-p)=a(p-n)=0$.  Consequently, every term in the sum over
$p\ge 3$ vanishes, and~\eqref{eq:tilingAtn} reduces to
\begin{equation*}%\label{eq:anm2one}
    a(n-2)=1.
\end{equation*}

As $n$ ranges over all negative odd integers with $n\le 3-2H$, the
values $n-2$ exhaust all sufficiently negative odd integers. Hence $a(k)=1$ for every sufficiently negative odd integer $k$.

By \eqref{eq:cond2}, it follows that
\begin{equation}\label{eq:akone}
    a(k)=1\quad\text{for every sufficiently large positive odd integer }k.
\end{equation}

We now evaluate~\eqref{eq:cond4} at $n=0$:
\begin{equation}\label{eq:tilingAt0}
    1
    =\sum_{p\in\Pp}a(-p)
    =\sum_{p\in\Pp}a(p).
\end{equation}

It is well-known that there are infinitely many primes, and by~\eqref{eq:akone} we have
$a(p)=1$ for every sufficiently large odd prime $p$.  Therefore the
nonnegative series on the right-hand side of~\eqref{eq:tilingAt0}
contains infinitely many terms equal to~$1$ and diverges to $+\infty$,
contradicting~\eqref{eq:tilingAt0}.

The contradiction shows that $E$ cannot be finite, i.e., there are
infinitely many positive even integers that do not belong to $\Pp-\Pp$.
\end{proof}

Theorem~\ref{thm:obstruction} immediately gives the following corollary.

\begin{coro}\label{cor:necessary}
If there exists a Borel probability measure $\mu$ on $\R$ for which
$E(\Pp)$ is an orthonormal basis of $L^2(\mu)$, then there are
infinitely many positive even integers that cannot be represented as
a difference of two primes.
\end{coro}

\begin{proof}
Apply Proposition~\ref{prop:tiling} and
Theorem~\ref{thm:obstruction} to
$a(k)=|\widehat{\mu}(k)|^2$.
\end{proof}

\begin{rem}\label{rem:strength}
Corollary~\ref{cor:necessary} is stronger than the assertion that the
Polignac conjecture must fail. It states that if the primes
were a spectrum, then infinitely many positive even integers would
fail to occur even once as a difference of two primes. By contrast,
the Polignac conjecture predicts that every positive even
integer occurs infinitely often as the difference of two consecutive
primes.
\end{rem}

\subsection{Proofs of the main results}

We now prove Theorem~\ref{thm:main} and Corollary~\ref{cor:polignac}.

\begin{proof}[Proof of Theorem~\ref{thm:main}]
Suppose that a Borel probability measure $\mu$ exists such that
$E(\Pp)$ is an orthonormal basis of $L^2(\mu)$. Define
$a(k)=|\widehat{\mu}(k)|^2$ for $k\in\Z$.
By Proposition~\ref{prop:tiling}, the function $a$ satisfies
\eqref{eq:cond1}--\eqref{eq:cond4}.

Under the eventual even difference hypothesis~\eqref{eq:ED}, every
sufficiently large positive even integer belongs to $\Pp-\Pp$. This
contradicts Theorem~\ref{thm:obstruction}. Hence no such
measure $\mu$ exists.
\end{proof}

\begin{proof}[Proof of Corollary~\ref{cor:polignac}]
Under the Polignac conjecture, every positive even integer
is the difference of two primes; indeed, it is the difference of
consecutive primes in infinitely many ways. Thus~\eqref{eq:ED} holds
with $H=1$. The conclusion follows immediately from
Theorem~\ref{thm:main}.
\end{proof}

Thus, under the eventual even difference hypothesis, the
incompatibility between the Fourier zero set forced by prime
differences and the exact Parseval tiling identity forced by
completeness precludes the prime numbers from being a spectrum
of any Borel probability measure. In particular, if the
Polignac conjecture holds, no such measure exists.

\subsection{A direct formulation under the Polignac conjecture}

For completeness, we give the transparent form of the
argument when every positive even integer is a prime difference.
This yields a self-contained statement whose proof is immediate from
the arguments developed above.

\begin{prop}\label{prop:direct}
Assume that $2\N\subseteq\Pp-\Pp$.
Then there is no nonnegative even function $a:\Z\to[0,\infty)$
satisfying
\(
    a(0)=1,~
    a(p-q)=0\;\;(p,q\in\Pp,\;p\neq q)\) and
\begin{equation}\label{eqidded}
\sum_{p\in\Pp}a(n-p)=1\;\;(n\in\Z).
\end{equation}
\end{prop}

\begin{proof}
We prove it by contradiction. The hypothesis $2\N\subseteq\Pp-\Pp$ and the zero-set condition imply
$a(2m)=0$ for all $m\in\Z\setminus\{0\}$.
Let $n<0$ be odd. In
\[
    1
    =
    a(n-2)
    +
    \sum_{\substack{p\in\Pp\\p\geq3}}a(n-p),
\]
the difference $n-p$ is a nonzero even integer for every odd prime
$p\geq3$. All terms in the sum therefore vanish, yielding
$a(n-2)=1$.
By symmetry, we conclude that $a(k)=1$ for every positive odd integer $k\geq3$.
Consequently,
\[
    \sum_{p\in\Pp}a(-p)
    =
    \sum_{p\in\Pp}a(p)
    =
    \infty,
\]
contradicting \eqref{eqidded} at $n=0$.
\end{proof}

\section{Concluding remarks and open problems}
Theorem~\ref{thm:main} shows that, under the hypothesis (ED), the
prime numbers $\Pp$ cannot form a spectrum of  any probability measure.  A natural weakening is to ask whether $\Pp$ can serve as a frame spectrum.  Recall that a
countable set $\Lambda$ is called a {\it frame spectrum} of $\mu$ if
$E(\Lambda)$ is a frame for $L^2(\mu)$, i.e., there exist positive constants $A$ and $B$ such that for any $f\in L^2(\mu)$,
\[
    A\|f\|^2\le \sum_{\lambda\in\Lambda}|\langle f,e_\lambda\rangle|^2
    \le B\|f\|^2.
\]
Every spectrum is a frame spectrum with $A=B=1$, but the
converse is not true in general.  The proof of Theorem~\ref{thm:obstruction}
relies crucially on the equality in the tiling
condition~\eqref{eq:cond4}, which is a consequence of Parseval's
identity for an orthonormal basis.  If $\Pp$ were merely a frame
spectrum, condition~\eqref{eq:cond4} would be replaced by the
following inequality
\[
    A \le \sum_{p\in\Pp} a(n-p) \le B \qquad(n\in\Z),
\]
and the counting argument that forces $a(k)=1$ for all large
odd $k$ no longer applies.  This motivates us the following question:

\begin{ques}\label{ques:frame}
Under the hypothesis (ED), does there exist a singular continuous Borel
probability measure $\mu$ for which $E(\Pp)$ is a frame (but not an
orthonormal basis) of $L^2(\mu)$?
\end{ques}

Moreover, we have the following quesiton concerning the necessity of the hypothesis (ED):
\begin{ques}\label{ques:unconditional}
Is it possible to prove that \(\mathbb{P}\) cannot be a spectrum of any Borel probability measure, without invoking the hypothesis (ED)?
\end{ques}

\section*{Acknowledgements}
The authors were supported by the National Natural Science Foundations of China (12271534, 12301105) and the University Research Project of Guangzhou Education Bureau (2024312332). The first author would like to thank Wei-Lin Zhang for discussion on the Polignac conjecture.

% ======================================================================

\end{document}